\documentclass[11pt]{article}
\usepackage{amsmath,amssymb,amsthm, graphicx}
\usepackage[dvips]{epsfig}
\usepackage{amsfonts}

\newtheorem{theorem}{Theorem}[section]
\newtheorem{proposition}[theorem]{Proposition}
\newtheorem{corollary}[theorem]{Corollary}

\newtheorem{problem}[theorem]{Problem}
\theoremstyle{definition}

\theoremstyle{remark}
\newtheorem{example}[theorem]{Example}

\newcommand{\R}{\mathbb{R}}
\newcommand{\Q}{\mathbb{Q}}
\newcommand{\N}{{\sf N}}
\newcommand{\C}{\mathcal{C}}
\newcommand{\D}{\sf D}
\newcommand{\E}{{\sf E}}
\newcommand{\K}{{\sf K}}

\renewcommand{\L}{{\sf{L}}}
\newcommand{\F}{\mathcal{F}}
\newcommand{\G}{\mathcal{G}}
\newcommand{\M}{\sf M}
\renewcommand{\P}{\sf P}
\newcommand{\algr}{{\hbox{\textsc{$d$-Representability}}}}
\newcommand{\algc}{{\hbox{\textsc{$d$-Collapsibility}}}}
\newcommand{\algl}{{\hbox{\textsc{$d$-LerayNumber}}}}

\DeclareMathOperator{\conv}{conv}

\begin{document}

\title{Intersection patterns of convex sets via simplicial complexes, a survey}
\author{Martin Tancer\thanks{
Department of Applied Mathematics and Institute for Theoretical Computer Science (supported by project 1M0545
of The Ministry of Education of the Czech Republic), Faculty of Mathematics
and Physics, Charles University, Malostransk\'e n\'am.~25, 118~00 Prague,
Czech Republic. Partially supperted by project GAUK 421511.}}
%Partially supported by project GAUK 49209.
%This work was done during the author's stay at KTH, Stockholm; and the author
%would like to thank to Anders Bj\"orner for this opportunity.
%E-mail: {\tt tancer@kam.mff.cuni.cz}
%}}

\maketitle

%\begin{center}
%{\Large
%MANUSCRIPT IN PREPARATION
%}
\begin{abstract}
The task of this survey is to present various results on intersection patterns
of convex sets. One of main tools for studying intersection patterns is a point
of view via simplicial complexes. We recall the definitions of so called
$d$-representable, $d$-collapsible and $d$-Leray simplicial complexes which are
very useful for this study. We study the differences among these notions and
we also focus on computational complexity for recognizing them. A list of
Helly-type theorems is presented in the survey and it is also discussed how
(important) role play the above mentioned notions for the theorems. We also
consider intersection patterns of good covers which generalize collections of
convex sets (the sets may be `curvy'; however their intersections cannot be too
complicated).  We mainly focus on new results.
\end{abstract}

%\end{center}

%AMS Subject Classification: 05E45 

%\begin{abstract}
%We introduce a notion of strong $d$-collapsibility. Using this notion, we
%simplify the proof of Matou\v{s}ek and the author~\cite{matousek-tancer08} showing that the nerve of a
%family of sets of size at most $d$ is $d$-collapsible.

%\end{abstract}

\section{Introduction}
An important branch of combinatorial geometry regards studying intersection
patterns of convex sets. Research in this area was initiated by a theorem of
Helly~\cite{helly23} which can be formulated as follows: If $C_1, \dots, C_n$ are convex sets
in $\R^d$, $n \geq d + 1$, and any collection of $d+1$ sets among $C_1, \dots,
C_n$ has a nonempty intersection, then all the sets have a common point. We
will focus on results of similar spirit; however, we have to set up some
notation first.

\subsection{Simplicial complexes} 
First we recall simplicial complexes which provide a convenient language for
studying intersection patterns of convex sets. We assume that the reader is
familiar with simplicial complexes, we only briefly mention the basics. For
further details the reader is referred to the books
like~\cite{hatcher01,matousek03, munkres84}.

We deal with finite \emph{abstract simplicial complexes}, i.e., collections
$\K$ of subsets of a finite set $X$ such that if $\alpha \in \K$ and $\beta
\subset \alpha$, then $\beta \in \K$. Elements of $\K$ are \emph{faces} of $\K$.
The \emph{dimension} of a face $\alpha \in \K$ is defined as $|\alpha| - 1$; 
$i$-dimensional faces for $i \in \{0,1,2\}$ are \emph{vertices}, \emph{edges}
and \emph{triangles} respectively. The \emph{dimension} of a simplicial complex
is the maximum of dimensions of its faces. Graphs coincide with 1-dimensional
simplicial complexes. If $V'$ is a subset of vertices of $\K$ then the \emph{induced
subcomplex} $\K[V']$ is a complex of faces $\alpha \in \K$ such that $\alpha
\subseteq V'$. We use the notation $\L \leq \K$ for pointing out that $\L$ is an
induced subcomplex of $\K$. An \emph{$m$-skeleton} of $\K$ is a simplicial complex consisting
of faces of $\K$ of dimension at most $m$. We denote it by $\K^{(m)}$. The
$m$-dimensional \emph{full simplex}, $\Delta_m$, is a simplicial complex with
the vertex set $\{1, \dots, m+1\}$ and all possible faces.
Let $Y$ be a set of affinely independent points in $\R^{|X|-1}$ such that
there is a bijection $f\colon X \rightarrow Y$. The \emph{geometric realization} of $\K$, denoted by
$|\K|$, is the topological space $\bigcup \{\conv f(\alpha)\colon \alpha
\in \K\}$, where $\conv$ denotes the convex hull.

\subsection{$d$-representable complexes}
Let $\C$ be a collection of some subsets of a given set $X$. The \emph{nerve}
of $\C$, denoted by $\N(\C)$, is a simplicial complex whose vertices are the
sets in $\C$ and whose faces are subcollections $\{C_1, \dots, C_k\} \subseteq
\C$ such that the intersection $C_1 \cap \cdots \cap C_k$ is nonempty. The
notion of nerve is designed to record the `intersection pattern' of the
sets in $\C$.

A simplicial complex $\K$ is \emph{$d$-representable} if it is isomorphic to the
nerve of a finite collection of convex sets in $\R^d$. Such a collection of
convex sets is called a \emph{$d$-representation} for $\K$. $d$-representable
simplicial complexes are the central objects of our study in this survey. As it
was mentioned above, they exactly record all possible intersection patterns of
finite collections of convex sets in $\R^d$.

Using the notion of $d$-representability, the Helly theorem can be reformulated
as follows: If a $d$-representable complex on at least $d+1$ vertices contains
all possible $d$-faces, then it is already a full simplex. This statement is, via
induction, equivalent with the following: $d$-representable simplicial complex does not contain an
induced $k$-dimensional \emph{simplicial hole} for $k \geq d$, i.e., a complex
isomorphic to $\Delta^{(k)}_{k+1}$. (Note that ``hole'' refers here to a hole
in a certain topological space, and in particular this notion is different from
$k$-hole in the context of Horton sets. The dimension of the hole refers to the
dimension of the boundary rather than to the dimension of the missing part. We
also remark that a geometric representation of
$k$-dimensional simplicial hole is homeomorphic to the \emph{$k$-sphere} $S^k$;
and it is the simplest way to obtain the $k$-sphere as a simplicial
complex.) The \emph{Helly number} of a simplicial complex $\K$ is 
the total number of vertices of the largest simplicial hole in $\K$ (i.e., the
dimension of the hole plus 2). Another reformulation of the Helly theorem thus
states that the Helly number of a $d$-representable complex is at
most $d+1$.

\begin{example}
Figure~\ref{f:2rep} shows a collection $\C = \{C_1, C_2, C_3, C_4, C_5\}$ of
convex sets (on left) and their nerve (on right). In other words the simplicial
complex on right is $2$-representable and $\C$ is a $2$-representation of
it. The Helly number of this collection equals $3$.
\end{example}

\begin{figure}
\begin{center}
\includegraphics{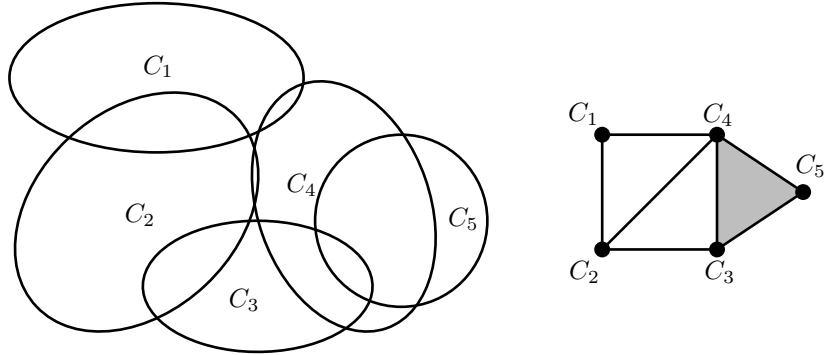}
\caption{A $2$-representable complex and its nerve.}
\label{f:2rep}
\end{center}
\end{figure}

\subsection{What is in the survey?} 
The task of the survey is to give an overview
of a recent developments on the study of intersection patterns of finite
collections of convex sets. We are mainly focusing on the description of
intersection patterns via simplicial complexes. Apart from geometrical `ad
hoc' arguments there are two other important approaches we are going to
discuss. First of them is combinatorial and regards $d$-collapsibility. The
second one is topological and regards the Leray number of a simplicial complex.
These two approaches are mainly discussed in the following two sections.
Algorithmic aspects of recognition of intersection patterns are briefly
discussed in section~\ref{s:algo}. In section~\ref{s:gc} we mention some
properties of good covers as a natural generalization of collections of convex
sets. Finally, section~\ref{s:thm} contains a list of theorems on intersection
patterns. This list is at the end of the survey in order that all the necessary
terminology is already built up. However, many of the results in
section~\ref{s:thm} can be understand without a detailed study of the previous
sections.

Since our task is to cover only a very selected part of convex geometry, we
refer the reader to other sources regarding the related areas. In particular we
refer to~\cite{matousek02} and the references therein for extended basic
overview on \emph{discrete convex geometry} including Radon,
Carath\'{e}odory, and Tverberg type theorems; we also refer to~\cite{eckhoff93}
for another point of view on the area; and to~\cite{goodman-pollack-wenger93}
for results on transversals to convex sets. We also do not focus on the theory
of $f$-vectors. For a reader interested in $f$-vectors we refer
to~\cite{billera-bjorner97} or to~\cite{kalai02} for a useful method for
investigating $f$-vectors (and related also to other branches mentioned here).

\section{$d$-collapsible and $d$-Leray complexes}
There are two other important classes of simplicial complexes related to the
$d$-representable ones. Informally, a simplicial complex is $d$-collapsible if
it can be vanished by removing faces of dimension at most $d-1$ which are
contained in a single maximal face; a simplicial complex is $d$-Leray if its
induced subcomplexes do not contain, homologically, holes of dimension $d$ or
more.

Wegner~\cite{wegner75} proved that $d$-representable simplicial complexes are
$d$-collapsible and also that $d$-collapsible complexes are $d$-Leray. 

Now we precisely define $d$-collapsible complexes and then $d$-Leray complexes.

Let $\K$ be a simplicial complex. Let $T$ be the collection of inclusion-wise
maximal faces of $\K$. A face $\sigma$ is \emph{$d$-collapsible} if there is only one
face $\tau \in T$ containing $\sigma$ (possibly $\sigma = \tau$), and moreover $\dim \sigma \leq d-1$. The simplicial complex
$$
\K' := \K \setminus \{ \eta \in \K: \eta \supseteq \sigma\}
$$
is an \emph{elementary $d$-collapse} of $\K$. For such a situation, we use the
notation $\K \rightarrow \K'$. A simplicial complex is \emph{$d$-collapsible} if
there is a sequence,
$$
\K \rightarrow \K_1 \rightarrow \K_2 \rightarrow \cdots \rightarrow \emptyset,
$$
of elementary $d$-collapses ending with an empty complex.

\begin{figure}
\begin{center}
\includegraphics{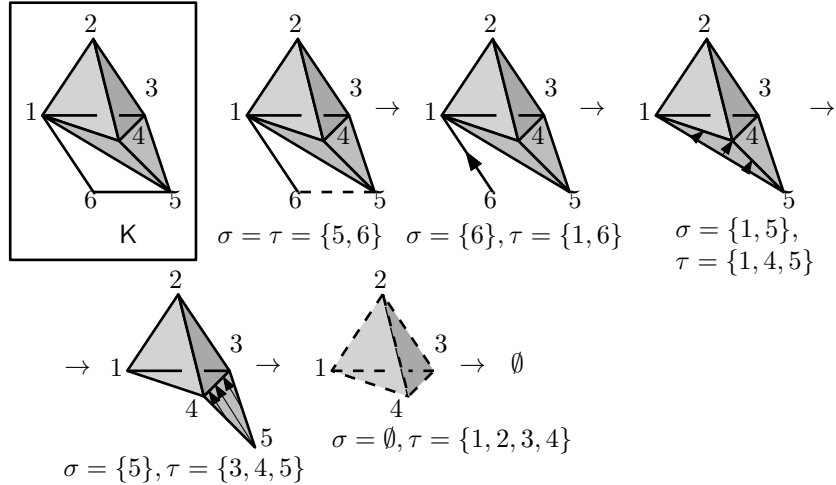}
\caption{A $2$-collapsing of a simplicial complex.}
\label{f:2col}
\end{center}
\end{figure}

\begin{example}
A simplicial complex $\K$ consisting of a full tetrahedron, two full triangles
and one hollow triangle in Figure~\ref{f:2col} is $2$-collapsible. For a
proof there is a 2-collapsing of $\K$ drawn on the picture. In every step the
faces $\sigma$ and $\tau$ are indicated.
\end{example}

A simplicial $\K$ complex is \emph{$d$-Leray} if the $i$th reduced homology
group $\tilde{H}_i(\L)$
(over~$\Q$) vanishes for every induced subcomplex $\L \leq \K$ and every $i
\geq d$.

We mention several remarks regarding $d$-collapsible and $d$-Leray complexes.
Deeper properties of them are studied in the following sections.

\begin{figure}
\begin{center}
\includegraphics{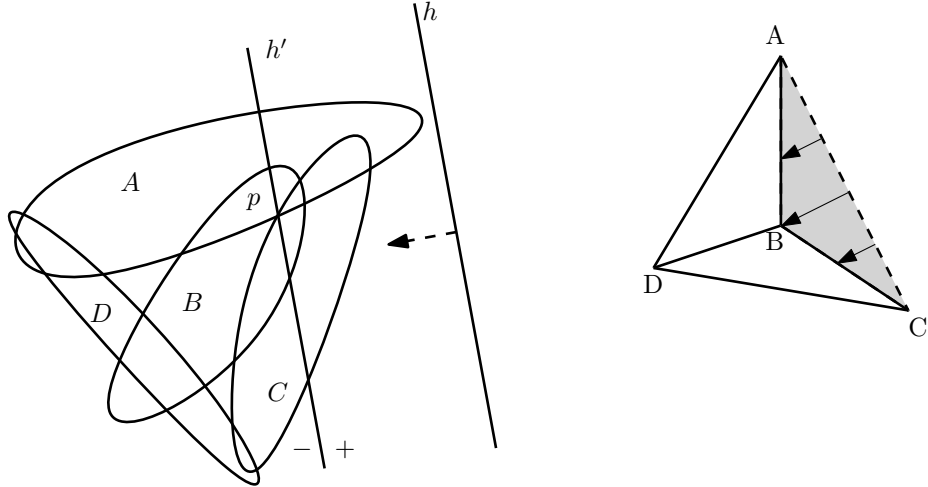}
\caption{A schematic sketch of the proof of Wegner's theorem. A generic
hyperplane $h$ is slided from infinity to minus infinity until there is a
nontrivial intersection of the convex sets on its positive side. In this case
it slides to $h'$ and it cuts off $A \cap B \cap C$ (it also cuts off $A \cap
C$, but for the moment we consider a maximal collection). From genericity there
is a single point $p \in A \cap B \cap C \cap h'$. It can be shown (using
Helly's theorem) that there is only at most $d$ sets of the starting collection
necessary to obtain $p$. In this case $\{p\} = A \cap C \cap h'$. Thus we obtain a
$d$-collapse with $\sigma = \{A, C\}$ and $\tau = \{A,B,C\}$. Finally, $A\cap
(h')^-, \dots, D\cap(h')^-$ form a $d$-representation for the resulting
collapsed complex thus the procedure can be repeated.} 
\label{f:wc}
\end{center}
\end{figure}

\begin{itemize}

\item
The fact that $d$-collapsible complexes are $d$-Leray is simple (for a reader
familiar with homology) since $d$-collapsing does not affect homology of
dimension $d$ or more.

It is a bit less trivial to show that a $d$-representable complex $\K$ is
$d$-collapsible. The idea is to slide a generic hyperplane (from infinity to
minus infinity) over a $d$-representation for $\K$ and gradually cut off whatever
is on the positive side of the hyperplane. See Figure~\ref{f:wc} and the text
bellow the picture for a more detailed sketch. The reader is referred
to~\cite{wegner75} for full details.

The
inclusion of $d$-representable complexes in $d$-Leray complexes can be also
deduced, without using Wegner's results, from the nerve theorem (see
Theorem~\ref{t:nerve}). 

\item
A $d$-dimensional simplicial complex is $(d+1)$-collapsible and hence also
$(d+1)$-Leray. For a complex $\K$ the smallest possible $\ell$ such that $\K$
is $d$-Leray is traditionally called the \emph{Leray number} of $\K$.

\item
Neither $d$-representability, $d$-collapsibility nor the Leray number is an
invariant under a homeomorphism: the full simplex $\Delta_m$ is
$0$-represent\-able; however, its barycentric subdivision is not even
$(m-1)$-Leray, since it contains an $(m-1)$-sphere as an induced subcomplex.

\item
It is not very difficult to see that every induced subcomplex of a
$d$-collapsible complex is again $d$-collapsible. If $\K[V'] \leq \K$ and 
$\K \rightarrow \K_1 \rightarrow \cdots \rightarrow \emptyset$ is a
$d$-collapsing of $\K$, then $\K[V'] \rightarrow \K_1[V'] \rightarrow \cdots
\rightarrow \emptyset = \emptyset[V']$ is a $d$-collapsing for $\K[V']$, where
some steps are possibly trivial, i.e., $\K_i[V'] = \K_{i+1}[V']$.

\item

The Helly theorem easily follows from the fact that $d$-represent\-able complexes
are contained in $d$-collapsible ones (or $d$-Leray ones). For we have that a
$d$-dimensional simplicial hole is neither $d$-collapsible (nor $d$-Leray).

On the other hand these two notions provide (much) stronger limitations to
intersection patterns than the Helly theorem. For instance they also exclude
(in dimension 2) the boundary of the octahedron (i.e., the simplicial complex
with vertices $\{-3,-2, -1, 1, 2, 3\}$ and faces $\alpha$ such that 
there is no $i \in \{1,2,3\}$ with $-i, i \in \alpha$) or a triangulation 
of a torus.
\end{itemize}

The gaps among these notions are discussed in more detail in the following section.

\section{Gaps among the notions}
\label{s:gaps}
In this section we overview how the notions of $d$-representable,
$d$-collapsible and $d$-Leray complexes differ. We also relate these
differences with the dimension of the complex.

\subsection{Every finite simplicial complex is $d$-representable for $d$ big
enough}
Let $\K$ be a simplicial complex on vertex set $\{1, \dots, n\}$. Let
$x_1, \dots, x_n$ be affinely independent points in $\R^{n-1}$ (i.e., they form
a simplex). For~a~non\-empty face $\alpha = \{a_1, \dots, a_t\} \in \K$ let $b_\alpha$ be the barycentre of the points $x_{a_1}, \dots, x_{a_t}$. Then for $i \in \{1,
\dots, n\}$ we set $C_i := \conv\{ b_\alpha: i \in \alpha, \alpha \in \K\}$. The reader is
welcome to check that sets $C_{i_1}, \dots, C_{i_k}$ intersect if and only if
$\{i_1, \dots, i_k\} \in \K$. Thus, the nerve of $C_1, \dots, C_n$ is
isomorphic to $\K$. See Figure~\ref{f:dr} for an illustration. (If we really
would not care about the dimension, it would be even easier to check the
situation where the points $b_{\alpha}$ are set to be the vertices of a simplex
of dimension $|\K| - 1$.)

\begin{figure}
\begin{center}
\includegraphics{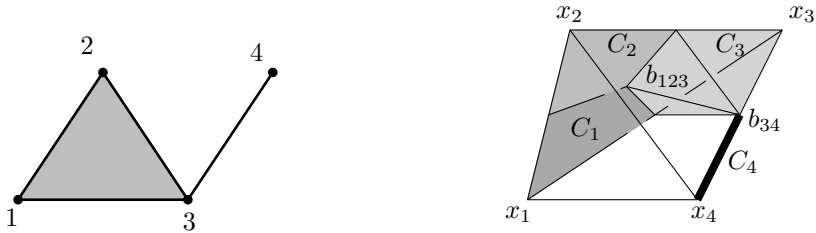}
\caption{Representing a complex.}
\label{f:dr}
\end{center}
\end{figure}

There is, however, another way how to obtain a representation of a complex
depending on the dimension of the complex. 

\begin{theorem}[Wegner~\cite{wegner67}, Perel'man~\cite{perelman85}]
\label{t:2d+1}
Let $\K$ be a $d$-dimensional simplicial complex. Then $\K$ is $(2d +
1)$-representable.
\end{theorem}

The value $2d + 1$ in Theorem~\ref{t:2d+1} is the least possible.
For example, the barycentric subdivisions of the $d$-skeleton of
a $(2d+2)$-dimensional simplex is not $2d$-representable. Case $d=1$ was
established by Wegner~\cite{wegner67}; general case is proved
in~\cite{tancer11dimensionarxiv}.

The references for Theorem~\ref{t:2d+1} are due to Eckhoff~\cite{eckhoff93}.
(Perel'man rediscovered Wegner's result.) Unfortunately, I have not been able
to check these sources in detail (the first one is in German, the second one is
in Russian). Thus I rather supply an idea of a proof (communicated by Ji\v{r}\'{i} Matou\v{s}ek).

\begin{proof}[Sketch of a proof of Theorem~\ref{t:2d+1}]
Let $\K$ be a $d$-representable complex with $n$ vertices.

A \emph{$k$-neighborly} polytope is a convex polytope such that every $k$
vertices form a face of the polytope. It is well known that there are
$2k$-dimensional $k$-neighborly polytopes with arbitrary number of vertices
for every $k \geq 1$. For instance \emph{cyclic polytopes} satisfy this property.
(See, e.g.,~\cite{matousek02} for a background on convex polytopes including
cyclic polytopes.) 

Let $Q$ be a $(2d + 2)$-dimensional $(d+1)$-neighborly polytope with $n$ vertices. Let $Q^*$ be a
polytope dual to $Q$. It has $n$ facets and any $d+1$ of its facets share
a face of the polytope. Finally, we consider the \emph{Schlegel diagram} of
$Q^*$. The Schlegel diagram of an $m$-dimensional convex polytope is a
projection of the polytope to $(m-1)$-space through a point beyond one of its
facets (the point is very close to the facet). In particular the facets of
$Q^*$ project to convex sets $C_1, \dots, C_n$ in $\R^{2d + 1}$ such that each
$d+1$ of them share the projection of a face of $Q^*$ (on their boundary).
Thus if we look at the nerve $\N$ of $C_1, \dots, C_n$, then it contains full
$d$-skeleton of a simplex with $n$ vertices. Therefore, without loss of generality, we can assume that $\K$ is a subcomplex of $\N$. 
Let $\vartheta = \{C_{i_1}, \dots, C_{i_j}\}$ be a face of $\N$ which does not
belong to $\K$. The sets $C_{i_1}, \dots, C_{i_j}$ intersect on their
boundaries and it is possible to remove their intersection by removing a small
neighborhood of $C_{i_1} \cap \cdots \cap C_{i_j}$ in each of the sets while
keeping the sets convex. Hence only $\vartheta$ and the superfaces of $\vartheta$
disappear from the nerve during this procedure. After repeating the procedure
we obtain a collection of convex sets with the nerve $\K$.

\end{proof}

%The value $2d + 1$ in Theorem~\ref{t:2d+1} is the best possible. This was shown
%in~\cite{matousek-tancer09}. We reproduce some of the details.
\subsection{The gap between representability and collapsibility}
For $d=0$ all three notions $0$-representable, $0$-collapsible and $0$-Leray
coincide and they can be replaced with `being a simplex'.

For $d=1$: $1$-representable complexes are clique complexes over interval
graphs; $1$-collapsible and $1$-Leray complexes are clique complexes over
chordal graphs (we remark that results in~\cite{lekkerkerker-boland62,
wegner75} easily imply these statements).

For $d \geq 2$ there is perhaps no simple characterization of
$d$-representable, $d$-collapsible and $d$-Leray complexes.
Wegner~\cite{wegner75} gave an example of complex, which is $2$-collapsible
but not $2$-representable. Matou\v{s}ek and the author~\cite{matousek-tancer09}
found $d$-collapsible complexes that are not $(2d-2)$-representable. Later, the
author~\cite{tancer10} improved this result by finding $2$-collapsible complexes that are not
$d$-representable (for any fixed $d$). We present some steps of both of the
constructions, since even the weaker construction contains some steps of their
own interest.

Let $\E$ be a $(d-1)$-dimensional simplicial complex which is not embeddable in
$\R^{2d-2}$. Such a complex always exist, for example the \emph{van Kampen
complex} $\Delta_{2d}^{(d-1)}$; see~\cite{vankampen32}, or the \emph{Flores
complex}~\cite{flores32}, which is the join of $d$ copies of a set of three
independent points. The first example is the nerve $\N(\E)$. It is
$d$-collapsible but not $(2d-2)$-representable due to the following two
propositions~\cite{matousek-tancer09}.

\begin{proposition}
Let $\K$ be a simplicial complex such that the nerve $\N(\K)$ is
$n$-representable. Then $\K$ embeds in $\R^n$, even linearly.
\end{proposition}

\begin{proposition}
Let $\F$ be a family of sets, each of size at most $n$. Then the nerve $\N(\F)$
is $n$-collapsible.
\end{proposition}

The second (stronger) example regards finite projective planes seen as
simplicial complexes. Let $(P, \mathcal L)$ be a finite projective plane, where
$P$ is the set of its points and $\mathcal L$ is the set of its lines. There is
a natural simplicial complex $\P$ associated to the projective plane. Its
ground set is $P$ and faces are the collections of points lying on a common
line.

It is not hard to show that $\P$ is $2$-collapsible. Non-representability of
$\P$ is summarized in the following theorem~\cite{tancer10}.

\begin{theorem}
\label{t:npp}
For every $d \in \mathbb N$ there is a $q_0 = q_0(d)$ such that if a complex $\P$ correspond to projective plane of order $q \geq q_0$ then $\P$ is not $d$-representable.
\end{theorem}

We also sketch a proof of Theorem~\ref{t:npp}. The reader is referred to the
original paper for full details. In contrast to the original paper we present it in an `inequality form'. For this we need few preliminaries.

One of important ingredients is a selection theorem by Pach~\cite{pach98}.

\begin{theorem}
\label{t:pach}
For every positive integer $d$ there is a constant $c = c(d) > 0$ with the
following property. Let $X \subset \R^d$ be a finite set of points in general
position. Then there is a point $a \in \R^d$ and disjoint disjoint subsets
$Z_1, \dots, Z_{d+1}$ of $X$, with $|Z_i| \geq c|X|$ such that the convex hull
of every transversal of $(Z_1, \dots, Z_{d+1})$ contains $a$.
\end{theorem}

We recall that a \emph{transversal} of a group of sets $(Z_1, \dots, Z_{d+1})$
is a set $\{z_1, \dots, z_{d+1}\}$ such that $z_i \in Z_i$. 

Now we let
$\mathcal S$ to be a subset of a simplicial complex $\K$ with a vertex set $V$
where 
$V = \{1, \dots, n\}$ and $|\mathcal S| = s$. For $\mathcal S' \subset \mathcal
S$ we define the \emph{deficiency} of $\mathcal S'$ as the number of vertices
of $\K$ which are not contained in any element of $\mathcal S'$, i.e., the value
$n - |\bigcup \mathcal S'|$. A value $\rho(k)$ of a
function $\rho\colon \{1, \dots, n \} \rightarrow \mathbb N$ is defined as the
maximum of deficiencies of sets $\mathcal S' \subseteq \mathcal S$ with $k$
elements. 
Then we
have the following inequality.

\begin{proposition}
\label{p:ineq}
If $\K$ is $d$-representable, then
$$
n - (d+1)\rho(c(d)s) \leq \dim \K + 1,
$$
where $c(d)$ is the Pach's constant.
\end{proposition}

\begin{proof}[Sketch of a proof.]
Let $C_1, \dots, C_n$ be the sets forming the $d$-representation of $\K$, set
$C_i$ corresponds to a vertex $i$. It can be assumed that these sets are open.
For every $\sigma \in \mathcal S$ there is a point $x_{\sigma}$ in the
intersection of all $C_i$ such that $i$ is a vertex of $\sigma$. Let $X =
\{x_{\sigma}: \sigma \in \mathcal S\}$. It can be assumed that $X$ is in
general position due to the openness of the sets $C_i$. So we have $Z_1, \dots,
Z_{d+1} \subseteq X$ and $a \in \R^d$ from Theorem~\ref{t:pach}. For a fixed $j
\in \{1, \dots, d+1\}$ the definition of $\rho$ implies that only $\rho(c(d)s)$ sets among $C_1, \dots, C_n$ can avoid the points of the set
$Z_j$. Thus there is at least $n - (d+1)\rho(c(d)s)$ of the $C_i$ that
meat all $Z_j$, and therefore they contain $a$. Hence the vertices of $\K$
corresponding to these $C_i$ form a face of $\K$ of dimension $n - (d+1)
\rho(c(d)s) - 1$.
\end{proof}

Theorem~\ref{t:npp} follows from Proposition~\ref{p:ineq} when $\mathcal S$ is
set of all maximal simplices of a projective plane $\P$. Then $n = s = q^2 + q +
1$, $\dim P = q + 1$, and $\rho(k) \leq (q^2 + q + 1)^{3/2}/k$ by a theorem
of Alon~\cite{alon85, alon86}.

\subsection{The gap between collapsibility and Leray number}
Wegner showed an example of complex which is $2$-Leray but not $2$-collapsible,
namely a triangulation $\D$ of the dunce hat. If we consider the multiple join
$\D \star \cdots \star \D$ of $d$ copies of $\D$, we obtain a complex which is
$2d$-Leray but not $(3d - 1)$-collapsible. See~\cite{matousek-tancer09} for
more details.

\section{Algorithmic perspective}
\label{s:algo}
As we consider different criteria for $d$-representability, it is also natural
to ask whether there is an algorithm for recognition $d$-representable
complexes. We denote this algorithmic question as \algr. More precisely, the
input of this question is a simplicial complex. The size of the input is the number
of faces of the complex. The value $d$ is considered as a fixed integer. The
output of the algorithm is the answer whether the complex is $d$-representable.

We can also ask similar questions for $d$-collapsible and $d$-Leray complexes
as relaxations of the previous problem. Thus we have algorithmic problems \algc
\ and \algl.

\subsection{Representability} 
The first mentioned problem \algr \ is perhaps the
most difficult among the three algorithmic questions. It is NP-hard for $d \geq
2$. Reduction can be done in a very similar fashion as a reduction for hardness of
recognition intersection graphs of segments~\cite{kratochvil-matousek89,
kratochvil-matousek94}. Full details can be found in \cite{tancer10np}. On the
other hand it is not hard to see that there is a PSPACE algorithm for \algr. It
is based on solving systems of polynomial inequalities.
See~\cite[Theorem 1.1(i)(a)]{kratochvil-matousek94} for a very similar reduction.

\subsection{Collapsibility} 
It is shown in~\cite{tancer10np} that \algc \ is
NP-complete for $d \geq 4$ and it is polynomial time solvable for $d \leq 2$.
For $d = 3$, the problem remains open.

\subsection{Leray number} 
The last question, \algl, is polynomial time solvable.
An equivalent characterization of $d$-Leray complexes is when induced
subcomplexes are replaced with \emph{links} of faces (including an empty face).
See~\cite[Proposition 3.1]{kalai-meshulam06} for a proof. The tests on links can be done in polynomial time since it is sufficient to test the homology up to the dimension of the complex.  

\subsection{Greedy collapsibility}
The algorithmic results above suggest that it is easier to test/compute the
Leray number than collapsibility. However, if we are interested in
them because of a hint for representability, computing collapsibility 
still can be more convenient, since $d$-collapsibility is closer to
$d$-representability than the Leray number. 
An example from section~\ref{s:gaps} is perhaps not so convincing;
however, there is a more important example. As it is shown in
section~\ref{s:gc}, $d$-collapsibility can distinguish collections of 
convex sets and good covers.

An useful tool for computation could be greedy $d$-collapsibility.
We say that a
simplicial complex $\K$ is \emph{greedily $d$-collapsible} if it is
$d$-collapsible and any sequence of $d$-collapses of $\K$ ends up in 
a complex which is still $d$-collapsible. In other words greedy
collapsibility allows us to collapse the faces of $\K$ 
in whatever order without risk of a bad choice. Thus, if a complex is greedily
$d$-collapsible, then there is a simple (greedy) algorithm for showing that it
is $d$-collapsible.
Not all $d$-collapsible complexes are greedily $d$-collapsible.
Complexes which are not greedily
$d$-collapsible for $d \geq 3$ are constructed in~\cite{tancer10np}. However, none of these complexes is
$d$-representable.
In summary there is a hope for obtaining a simple algorithm for showing that a
complex is either $d$-collapsible or it is not $d$-representable if the answer
to the following question is true.

\begin{problem}
Is it true that every $d$-representable simplicial complex is greedily
$d$-collapsible?
\end{problem}

\section{Good covers}
\label{s:gc}
%An open set in $\R^d$ homeomorphic to a ball is a \emph{$d$-cell}. 
A \emph{good cover}
in $\R^d$ is a collection of open sets in $\R^d$ such that the intersection of
any subcollection is either empty or a contractible (in particular, the sets in
the collection are contractible).\footnote{The definition of a good cover is
not fully standard in the literature. For example, it may be assumed that sets
in the collection are closed instead of open, or that the intersections are
homeomorphic to (open) balls instead of contractible. These differences are
not essential for the most of the purposes mentioned here, because all these
options satisfy the assumptions of a nerve theorem (see the text bellow).} We consider only finite good covers. A simplicial complex is \emph{topologically $d$-representable} if it is isomorphic to the nerve of a (finite) good cover in $\R^d$. We should emphasize that (for our purposes) a good cover need not cover whole $\R^d$.

Topologically $d$-representable complexes generalize $d$-representable
complexes since every collection of convex sets is a good cover.

%\subsection{Standard definition of good covers?}
%There is not a fully standard definition of a good cover in the literature. For
%instance it can appear that a good cover is a collection of open 
%contractible sets such that the intersection of every subcollection is again
%contractible or empty. For the purposes of this survey we call such a
%collection a \emph{weak good cover}. Clearly, a good cover is also a weak
%good cover.

%Let us focus in a bit more detail on various definitions of good covers. For
%motivation we recall that standard tricks yield to coincidence of various
%possible definitions of $d$-representable simplicial complexes.

%\begin{fact}
%Let $\K$ be a simplicial complex. Then the following conditions are equivalent.
%\begin{enumerate}
%\item $\K$ is $d$-representable;
%\item $\K$ is isomorphic to a nerve of a finite collection of convex polytopes;
%\item $\K$ is isomorphic to a nerve of a finite collection of compact convex
%sets;
%\item $\K$ is isomorphic to a nerve of a finite collection of open convex sets.
%\end{enumerate}
%\end{fact}

%Regarding good covers---the main focus is on the contractibility of
%intersections; however, we can also demand other conditions analogous to the
%above mentioned conditions. Unfortunately, it is not \emph{a priori} clear
%which of those possible definitions coincide. From a point of view of the
%author of this survey a reasonable notion of a good cover should generalize
%collections of convex sets; should have contractible intersections (if
%nonempty); and the nerve of such a good cover should be $d$-Leray (see the
%following text).

\subsection{Nerve theorems}
Suppose we are given a collection $\F$ of subsets of $\R^d$. If the sets are
``sufficiently nice'' and also all their intersections are sufficiently nice
then the nerve of the collection, $\N(\F)$, is homotopy equivalent to the union of
the sets in the collection, $\bigcup \F$. For a weaker assumption on
``sufficiently nice'', it is possible to derive not necessarily homotopy
equivalence, but at least equivalence on homology (up to some level). Such
results are known as \emph{homotopic/homological} nerve theorems.

We mention here a one of possible versions (suitable for our
purposes); see~\cite[Corollary 4G.3]{hatcher01}.  

\begin{theorem}[A homotopy nerve theorem]
\label{t:nerve}
Let $\F$ be a collection of open contractible sets in a paracompact space $X$
such that $\bigcup \F = X$ and every nonempty intersection of finitely many
sets in $\F$ is contractible (or empty). Then the nerve $\N(\F)$ and $X$ are
homotopy equivalent.
\end{theorem}

\begin{corollary}
\label{c:gcl}
The nerve of every good cover is $d$-Leray.
\end{corollary}

\subsection{Good covers versus collections of convex sets}
Good covers have many similar properties as collections of convex sets. Many
results on intersection patterns of convex sets can be generalized for good
covers. We will discuss these generalizations in the following section. An
exceptional case is Theorem~\ref{t:npp} which cannot be generalized for good
covers.

On the other hand it is not hard to see that topologically $d$-representable
complexes are strictly more general than $d$-representable complexes for $d
\geq 2$. There is a less trivial example on Figure~\ref{f:ce} showing that
there is a complex which is topologically $d$-representable but not
$d$-collapsible. Originally, Wegner conjectured that there is no such example.
See~\cite{tancer11counterex} for more details.

\begin{figure}
\begin{center}
\includegraphics{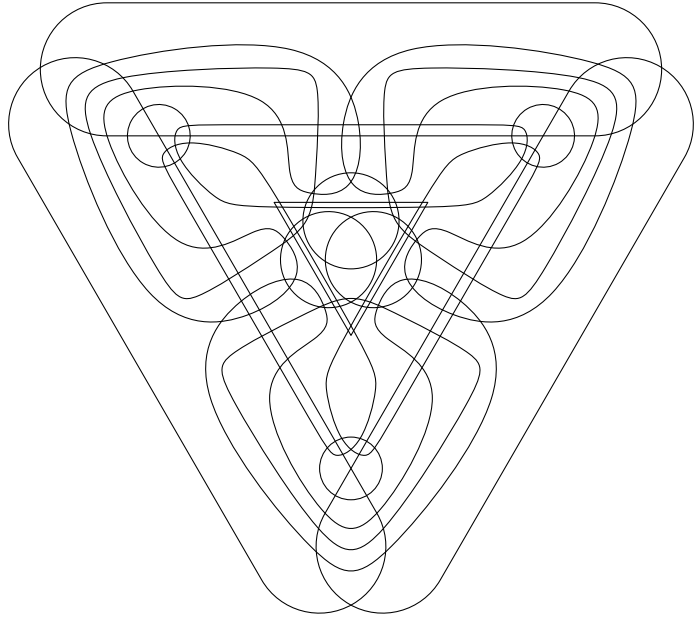}
\end{center}
\caption{A good cover $\F$ such that the nerve of $\F$ is even not
$d$-collapsible.}
\label{f:ce}
\end{figure}

In addition, this example also distinguishes good covers in $\R^2$ and
arrangements of convex sets in a topological plane:\footnote{This observation
is by Xavier Goaoc.} An \emph{arrangement of pseudolines} is a set of curves (called \emph{pseudolines}) in the plane such
that every two pseudolines intersect in exactly one point. The most convenient
way is perhaps to think of the plane as a subset of the real projective plane;
then we can allow even ``parallel'' pseudolines.  Such an arrangement can be
extended to a \emph{topological plane} where there is a pseudoline through
every pair of points. See, e.g.,~\cite{goodman-pollack-wenger-zamfirescu94} for
more precise definitions and another background. A \emph{convex set in a
topological plane} is such a subset that every two points are connected with a
``segment'' of a pseudoline. The nerve of a bounded collection of convex sets
in a topological plane is 2-collapsible. This can be shown in a very similar
way as Wegner's theorem on inclusion of $d$-representable complexes in
$d$-collapsible ones. Thus the example from Figure~\ref{f:ce} cannot be
a collection of convex sets in a topological plane.

%5$). (This is an issue of a work in progress.) If it really turns out being
%true, then it shows another major difference in between of $d$-representable and
%topologically $d$-representable complexes since $d$-representable complexes are
%recognizable.

\section{Helly type theorems}
\label{s:thm}

In this section we overview some Helly-type results on intersection patterns of
convex sets. We always start with geometrical formulation. Then we reformulate
such a result via $d$-representable simplicial complexes. 
We also discuss possible extensions to $d$-collapsible
or $d$-Leray complexes. (The former ones then have geometric consequences for good covers.)

\subsection{The Helly theorem.} 
For completeness of this section we also recall the
Helly theorem mentioned in the introduction.

We have the following geometric formulation.

\begin{theorem}[Helly, \cite{helly23}]
\label{t:hgeo}
If $C_1, \dots, C_n$ are convex sets
in $\R^d$, $n \geq d + 1$, and any collection of $d+1$ sets among $C_1, \dots,
C_n$ has a nonempty intersection, then all the sets have a common point.
\end{theorem}

A topological extension of the Helly theorem was proved few years later by Helly
himself~\cite{helly30}. His setting was for good covers. We present here a
setting for $d$-Leray complexes. Note that the theorem
stated here is trivial; however, the fact that it is meaningful relies on
Corollary~\ref{c:gcl}.

\begin{theorem}
\label{t:htopo}
The Helly number of a $d$-Leray simplicial complex is at most $d+1$. 
\end{theorem}

Theorem~\ref{t:hgeo} is a consequence of Theorem~\ref{t:htopo} if it is used
for $d$-repres\-ent\-able complexes.

\subsection{The colorful Helly theorem.}
The colorful Helly theorem regards the situation where convex sets are
colored. If there is enough color classes and every rainbow collection of the
colored sets contains a point in common, then the sets of a certain color class
contain a point in common.

\begin{theorem}[Colorful Helly, Lov\'asz~\cite{lovasz74}]
Let $\F_1, \dots, \F_{d+1}$ be families of convex sets in $\R^d$. Suppose that
for every choice $F_1 \in \F_1, \dots, F_{d+1} \in \F_{d+1}$ the intersection
$F_1 \cap \cdots \cap F_{d+1}$ is nonempty. Then there is $i \in [d+1]$ such
that the intersection of the sets in $\F_i$ is nonempty.
\end{theorem}

The Helly theorem is the consequence of the colorful Helly theorem if we set
$\F_1 = \F_2 = \cdots = \F_{d+1}$.

The reformulation via simplicial complexes is the following.

Let $V$ be a finite set partitioned into disjoint color classes $V_1, \dots,
V_k$. A subset $W \subseteq V$ is \emph{rainbow} if $|V_i \cap W| \leq 1$ for
$i \in [k]$.

\begin{theorem}
\label{t:crepre}
Let $\K$ be a $d$-representable simplicial complex with vertices partitioned
into $d+1$ color classes. Assume that every rainbow subset of vertices is a
simplex of $\K$. Then there is a color class 
such that its vertices form a simplex of $\K$.  
\end{theorem}

Let $\M'$ be a simplicial complex with the vertex set $V$ from above whose
faces are the rainbow subsets of $V$. It is not hard to see that $\M'$ is a
matroidal complex of rank $k$. 

Kalai and Meshulam~\cite{kalai-meshulam05} obtained the following matroidal
extension of the colorful Helly theorem.

\begin{theorem}
\label{t:cmatro}
Let $\K$ be a $d$-collapsible simplicial complex on $V$ and let $\M$ be a
matroidal complex on $V$ with rank function $\rho$ such that $\M \subseteq \K$. Then there is a simplex
$\alpha \in \K$ such that $\rho(\alpha) = \rho(\M)$ and $\rho(V \setminus
\alpha) \leq d$.
\end{theorem}

Theorem~\ref{t:cmatro} indeed generalizes Theorem~\ref{t:crepre}. If we set $\M =
\M'$, then $\alpha$ contains all vertices of a certain color class (since
$\rho(V \setminus \alpha) \leq d$) and moreover $\alpha$ even contains a vertex
of every color class (since $\rho(\alpha) = \rho(\M))$.  

More importantly, Kalai and Meshulam~\cite{kalai-meshulam05} obtained a
topological generalization (with a weaker conclusion, but still more general
then that of Theorem~\ref{t:crepre}).

\begin{theorem}
\label{t:ctopo}
Let $\K$ be a $d$-Leray simplicial complex on $V$ and let $\M$ be a
matroidal complex on $V$ with rank function $\rho$ such that $\M \subseteq \K$. Then there is a simplex
$\alpha \in \K$ such that  $\rho(V \setminus \alpha) \leq d$.
\end{theorem}

\subsection{The fractional Helly theorem.} Let $\C$ be again a collection of
convex sets in $\R^d$ (containing at least $d+1$ sets). The Helly theorem
assumes that if \emph{every} $d+1$-tuple has a nonempty intersection then
\emph{all} the sets have point in common. The fractional Helly theorem is
designed for a situation when \emph{many} $d+1$-tuples have a nonempty
intersection concluding that \emph{many} sets of the collection have a point 
in common.

\begin{theorem}[Fractional Helly, Katchalski and Liu~\cite{katchalski-liu79}]
For every $a \in (0,1]$ and $d \in \N$ there is $b = b(d, a) \in (0,1]$ with
the following property. Let $\C$ be a collection of $n$ convex sets
in $\R^d$ $(n \geq d+1)$. Assume that the number of $(d+1)$-tuples with a
nonempty intersection is at least $a \binom{n}{d+1}$. Then there is a point common to at least $bn$ sets in $\C$.
\end{theorem}

The largest possible value for $b(d,a)$ is $1 - (1-a)^{1/(d+1)}$ due to
Kalai~\cite{kalai84b} and Eckhoff~\cite{eckhoff85} (i.e., it is even known that
$b$ cannot be larger). We remark that the Helly theorem is a special case when
setting $a = 1$.

There is also a topological extension of the fractional Helly theorem by Alon,
Kalai, Matou\v{s}ek and Meshulam~\cite{alon-kalai-matousek-meshulam02} (with
the same bound for $b$). They actually prove a bit stronger result (in order to
obtain topological $(p,q)$-theorem); however, we prefer to avoid the technical
details and so we present the result only in this simpler form.

\begin{theorem}
For every $a \in (0,1]$ and $d \in \N$ there is $b = b(d, a) \in (0,1]$ with
the following property. Let $\K$ be a $d$-Leray complex with $n$ vertices $(n
\geq d+1)$. Assume that there are at least $a \binom{n}{d+1}$
$d$-faces in $\K$. Then there is a face of size at least $bn - 1$ in
$\K$.
\end{theorem}

The largest possible value for $b(d,a)$ is again $1 - (1-a)^{1/(d+1)}$.

\subsection{The $(p,q)$ theorem.} Let $p$, $q$ be integers such that $p \geq q
\geq d+1$. A family $\F$ of convex sets in $\R^{d}$ has the \emph{$(p,q)$
property} if among every $p$ sets of $\F$ some $q$ have a nonempty
intersection. The \emph{pinning number}, $\pi(\F)$, of a 
family $\F$ is the smallest number of points in $\R^d$ that intersect all
members of $\F$. 

\begin{theorem}[$(p,q)$-theorem, Alon and Kleitman~\cite{alon-kleitman92}]
For every $p \geq q \geq d+1$ there is a number $C = C(p,q,d)$ such that 
$\pi(\F) \leq C$ for every family of convex sets in $\R^d$ with $(p,q)$
property.
\end{theorem}

The $(p,q)$ theorem was originally conjectured by Hadwiger and Debrunner.

In order not to introduce a new symbol, we let $C$ denote the smallest constant
for which is the assertion of the $(p,q)$ theorem valid. The Helly theorem
simply says that $C(d+1,d+1,d) = 1$. In general, there are, however, big gaps
between lower and upper bounds for $C$. In the first nontrivial case it is
relatively easy to come up with an example showing $C(4,3,2) \geq 3$; see
Figure~\ref{f:43prop}. Kleitman, Gy\'{a}rf\'{a}s and
T\'{o}th~\cite{kleitman-gyarfas-toth01} proved that $C(4,3,2)
\leq 13$. The author of this survey believes that the actual value of
$C(4,3,2)$ is much closer to $3$ then $13$; however, it seems difficult to
obtain the precise value. It is widely open what is the value of $C(p,q,d)$ for
larger $p$, $q$ and $d$.

\begin{figure}
\begin{center}
\includegraphics{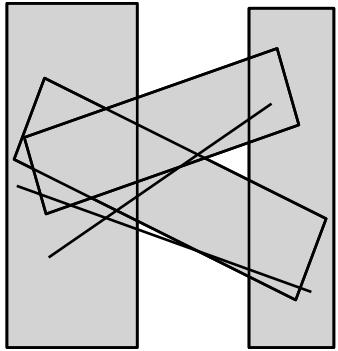}
\caption{Six convex sets with (4,3) property and pinning number 3.}
\label{f:43prop}
\end{center}
\end{figure}

Now we reformulate the setting for simplicial complexes. A simplicial complex
$\K$ has \emph{$(p,q)$ property} for $p \geq q$ if among every $p$ vertices of
$\K$ there is $q$ of them forming a face (of dimension $q - 1$). The
\emph{pinning number}, $\pi(K)$, of $\K$ is the smallest number of faces of
$\K$ such that every vertex of $\K$ is in at least one of these faces. Then the
statement of the $(p,q)$ theorem remains valid even for $d$-Leray complexes and
consequently for good covers due to Alon et
al.~\cite{alon-kalai-matousek-meshulam02}.\footnote{Precisely speaking, Alon et
al. state the theorem for good covers only; however, the same reasoning can be
used for $d$-Leray complexes.}

\begin{theorem}
For every $p \geq q \geq d+1$ there is a number $C' = C'(p,q,d)$ such that 
$\pi(\K) \leq C'$ for every $d$-Leray complex $\K$ with $(p,q)$ property.
\end{theorem}

\subsection{The Amenta theorem.}
In all previous cases we were considering properties of collections of convex
sets. Now we replace \emph{convex sets} with \emph{a finite disjoint union of
convex sets}. It turns out that there is also a Helly-type theorem for this
case (if we keep this property also for intersections).

Let $\F$ be a finite family of subsets of some ground set. A family $\G$ is
\emph{$(\F, k)$-family} if for every nonempty $\G' \subseteq \G$ the
intersection of elements of $\G'$ is the disjoint union of at most $k$ members
of $\F$. We have defined the Helly number only for simplicial complexes. For
purposes of this subsection we say that a family $\F$ has \emph{Helly number}
$h = h(\F)$ equal to the Helly number of the nerve of $\F$. (Here we allow $\F$
to be possibly infinite.)

\begin{theorem}[Amenta~\cite{amenta96}]
Let $\F$ be a finite family of compact convex sets in $\R^d$. Let $\G$ be an
$(\F, k)$-family. Then $h(\G) \leq k(d+1)$.
\end{theorem}

The bound $k(d+1)$ in Amenta's theorem is optimal as can be shown with the
following example. Let $\C = \{C_1, \dots, C_{d+1}\}$ be a collection of convex
sets in $\R^d$ such that every $d$ of them has nonempty intersection; however,
the intersection of the whole collection is empty (for example, $\C$ might be
the collection of facets of a $d$-simplex). Let $\C_* = \C \cup \{C\}$ where
$C$ is the convex hull of the union of sets in $\C$. Let us consider $k$
disjoint copies $\C_*^i = \{C^i_1, \dots, C_{d+1}^i, C^i\}$ of
$\C_*$ for $i \in \{1,\dots, k\}$ such that the sets $C^i$ are pairwise
disjoint. Now we construct sets $D^i_j$ for $i \in \{1, \dots, k\}$, $j \in
\{1, \dots, d+1\}$ by setting 
$$
D^i_j = \bigg(\bigcup\limits_{m \neq i} C^m\bigg) \cup C^i_j.
$$
Then $\G = \{D^i_j\}$ is an $(\F, k)$-family and $h(\G) = k(d+1)$. See
Figure~\ref{f:loweram}.

\begin{figure}
\begin{center}
\includegraphics{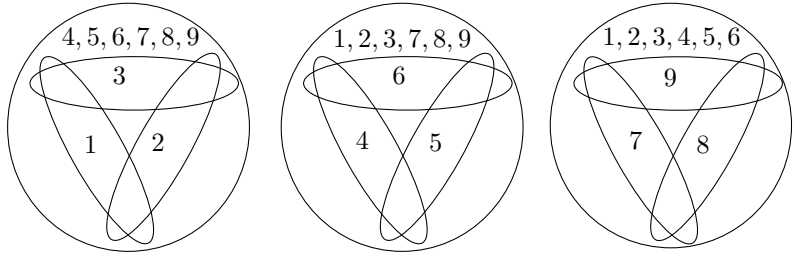}
\caption{The lower bound for the Amenta's theorem. The convex hulls are a bit
enlarged in order to make the figure more lucid. Moreover, the notation of
the sets is simplified.}
\label{f:loweram}
\end{center}
\end{figure}

Now we focus on a topological version. Due to the fact that we consider union
of sets, there is no simple statement for a topological generalization of Amenta's theorem using $d$-Leray complexes. Thus we prefer set up the statement for good covers in this case.

\begin{theorem}[Kalai, Meshulam~\cite{kalai-meshulam08}]
\label{t:km_a}
Let $\F$ be a finite good cover in $\R^d$. Let $\G$ be an $(\F,
k)$-family. Then $h(\G) \leq k(d+1)$.
\end{theorem}

Eckhoff and Nischke~\cite{eckhoff-nischke09} recently proved Amenta's theorem
in very abstract setting via Morris's pigeonhole principle. A (possibly
infinite) family $\F$ is \emph{intersectional}, if for any finite subfamily
$\F'$, the intersection $\cap \F'$ is either empty or belongs to $\F$. We call
$\F$ \emph{nonadditive} if for any finite subfamily $\F'$ of disjoint sets (at
least two nonempty), $\cup \F' \not\in \F$. The following theorem generalizes
the previous two theorems by setting $\F$ to be either a family of convex
compact sets, or a family of all intersections of a good cover.

\begin{theorem}
Let $\F$ be an intersectional and nonadditive set family. If $\G$ is an $(\F,
k)$ family, then $h(\G) \leq k h(\F)$.
\end{theorem}

From another point of view, a generalization of Theorem~\ref{t:km_a} for 
collections of sets in more general topological spaces
(than $\R^d$) is obtained by Colin de Verdi\`ere,
Ginot and Goaoc~\cite{colindeverdiere-ginot-goaoc11arxiv}. More importantly
their generalization applies to collections of sets that need not come from good covers. For instance it applies to a collection of two sets homeomorphic to a ball which intersect in two balls.

By $\Gamma$ we denote a
locally arc-wise connected topological space. Then $d_{\Gamma}$ is the smallest
integer such that every open subset of $\Gamma$ has a trivial
$\Q$-homology in dimension $d_\Gamma$ and higher.
An open subset of $\Gamma$ with singular $\Q$-homology equivalent to a point is
a \emph{$\Q$-homology cell}. A family of open subsets of $\Gamma$ is acyclic if
for every nonempty subfamily $\G \subset \F$, the intersection of $\G$ is a
disjoint union of $\Q$-homology cells. Colin de Verdi\`ere et al prove the
following result.

\begin{theorem}
\label{t:cgg}
Let $\F$ be a finite acyclic family of open subsets of locally arc-wise
connected topological space $\Gamma$. If any subfamily of $\F$ intersect in at
most $k$ connected components, then the Helly number of $\F$ is at most
$k(d_{\Gamma} + 1)$.
\end{theorem}

In particular, $d_{\R^d} = d$ and every good cover is acyclic, therefore
Theorem~\ref{t:cgg} indeed generalizes Theorem~\ref{t:km_a}. Let us also remark
that $d_{\Gamma} = d$ for a $d$-dimensional manifold which is either noncompact
or nonorientable and $d_{\Gamma} = d+1$ for a compact orientable manifold.

Theorem~\ref{t:cgg} can be furthermore generalized when the homology is zero
only from a certain dimension. For this case we already refer the reader
to~\cite{colindeverdiere-ginot-goaoc11arxiv}. On the other, the case when the
homology vanish for small dimensions is considered in (preceding)
paper~\cite{matousek97}. The bound to the Helly number is, however, much weaker.

\section*{Acknowledgement}
I thank Xavier Goaoc for discussions on new generalizations of Amenta's theorem
and additional remarks, Ji\v{r}\'{\i} Matou\v{s}ek for many valuable comments
to the preliminary version of the survey, and also Janos Pach for  
discussions about $d$-dimensional complexes which are not $2d$-representable.

\bibliographystyle{alpha}
%\bibliography{/home/martin/clanky/bib/grph,/home/martin/clanky/bib/topocom,/home/martin/clanky/bib/combgeo,/home/martin/clanky/bib/compl,/home/martin/clanky/bib/general}
\bibliography{/home/martin/clanky/bib/general}

\end{document}